\documentclass[oneside,british]{amsart}
\usepackage[T1]{fontenc}
\usepackage{babel}
\usepackage{amstext}
\usepackage{amsthm}
\usepackage{amssymb}
\usepackage[unicode=true,
 bookmarks=true,bookmarksnumbered=false,bookmarksopen=false,
 breaklinks=false,pdfborder={0 0 1},backref=false,colorlinks=false]
 {hyperref}
\hypersetup{pdftitle={Combining the Runge approximation and the Whitney embedding theorem in hybrid imaging},
 pdfauthor={Giovanni S. Alberti and Yves Capdeboscq}}

\makeatletter
\theoremstyle{plain}
\ifx\thechapter\undefined
	\newtheorem{thm}{\protect\theoremname}
\else
	\newtheorem{thm}{\protect\theoremname}[chapter]
\fi
\theoremstyle{definition}
\newtheorem{example}[thm]{\protect\examplename}
\theoremstyle{definition}
\newtheorem{defn}[thm]{\protect\definitionname}
\theoremstyle{remark}
\newtheorem{rem}[thm]{\protect\remarkname}
\theoremstyle{plain}
\newtheorem{cor}[thm]{\protect\corollaryname}
\theoremstyle{plain}
\newtheorem{prop}[thm]{\protect\propositionname}
\theoremstyle{plain}
\newtheorem{lem}[thm]{\protect\lemmaname}

\usepackage{empheq}
\usepackage[usenames,dvipsnames]{xcolor}
\usepackage{tikz}
\usepackage{standalone}
\usepackage{dsfont}

\usepackage{color}
\newcommand{\ga}[1]{{\textcolor{blue}{#1}}}

\makeatother

\providecommand{\corollaryname}{Corollary}
\providecommand{\definitionname}{Definition}
\providecommand{\examplename}{Example}
\providecommand{\lemmaname}{Lemma}
\providecommand{\propositionname}{Proposition}
\providecommand{\remarkname}{Remark}
\providecommand{\theoremname}{Theorem}

\begin{document}
\global\long\def\R{\mathbb{R}}%

\global\long\def\ii{\mathrm{i}}%

\global\long\def\N{\mathbb{N}}%

\global\long\def\C{\mathbb{C}}%

\global\long\def\supp{{\rm supp}}%

\global\long\def\span{\operatorname{span}}%

\global\long\def\ran{\operatorname{ran}}%

\global\long\def\rank{\operatorname{rank}}%

\global\long\def\div{\operatorname{div}}%

\global\long\def\H{\mathcal{H}}%

\global\long\def\k{\omega}%

\global\long\def\phi{\varphi}%

\global\long\def\epsilon{\varepsilon}%

\global\long\def\B{\mathcal{B}}%

\global\long\def\E{\mathcal{E}}%

\global\long\def\l{\ell}%

\global\long\def\nplusd{\ga{[\frac{d+n}{\alpha}]}}%

\global\long\def\e{\mathbf{e}}%

\global\long\def\one{\mathds{1}}%

\global\long\def\bo{\partial\Omega}%

\title[Runge approximation and Whitney embedding in hybrid imaging]{Combining the Runge approximation and the Whitney embedding theorem
in hybrid imaging}
\author{Giovanni S. Alberti}
\address{Department of Mathematics, University of Genoa, Via Dodecaneso 35,
16146 Genova, Italy}
\email{giovanni.alberti@unige.it}
\author{Yves Capdeboscq}
\address{Université de Paris, CNRS, Sorbonne Université, Laboratoire Jacques-Louis
Lions UMR7598, Paris, France}
\email{capdeboscq@ljll.univ-paris-diderot.fr}
\date{8\textsuperscript{th} May 2019}
\begin{abstract}
This paper addresses enforcing non-vanishing constraints for solutions
to a second order elliptic partial differential equation by appropriate
choices of boundary conditions. We show that, in dimension $d\geq2$,
under suitable regularity assumptions, the family of $2d$ solutions
such that their Jacobian has maximal rank in the domain is both open
and dense. The case of less regular coefficients is also addressed,
together with other constraints, which are relevant for applications
to recent hybrid imaging modalities. Our approach is based on the
combination of the Runge approximation property and the Whitney projection
argument {[}Greene and Wu, Ann.\ Inst.\ Fourier (Grenoble), 25(1,
vii):215--235, 1975{]}. The method is very general, and can be used
in other settings.
\end{abstract}

\maketitle

\section{Introduction}

We consider a general second-order elliptic equation
\begin{equation}
Lu:=-\div(a\nabla u+bu)+c\cdot\nabla u+qu=0\quad\text{in }\Omega,\label{eq:PDE}
\end{equation}
where $\Omega\subseteq\R^{d}$, $d\ge2$, is a bounded and smooth
domain. We assume that $L$ is uniformly elliptic, namely,
\begin{equation}
a\left(x\right)\xi\cdot\xi\geq\lambda\left|\xi\right|^{2},\qquad\text{ a.e. }x\in\Omega,\,\xi\in\R^{d},\label{eq:coercive}
\end{equation}
for some $\lambda>0$. The parameters of equation (\ref{eq:PDE})
are assumed to satisfy mild regularity assumptions, namely either
\begin{equation}
a\in C^{\l-1,\alpha}\left(\overline{\Omega};\mathbb{R}^{d\times d}\right),\,b\in C^{\l-1,\alpha}(\overline{\Omega};\R^{d}),\,c\in W^{\l-1,\infty}(\Omega;\R^{d}),\,q\in W^{\l-1,\infty}(\Omega;\R),\label{eq:reg}
\end{equation}
with $\ell\geq1$ and $\alpha\in(0,1)$, or 
\begin{equation}
a\in L^{\infty}\left(\Omega;\mathbb{R}^{d\times d}\right),\,b,c\in L^{\infty}\left(\Omega;\mathbb{R}^{d}\right),q\in L^{\infty}\left(\Omega;\mathbb{R}^{d}\right),\label{eq:reg0}
\end{equation}
which will be referred to as $\ell=0$. By classical elliptic regularity
theory \cite{GILBARG-2001,Giaquinta2005,morrey-2008}, the solutions
to (\ref{eq:PDE}) belong to $C_{\mathrm{loc}}^{\l,\alpha}\left(\Omega;\R\right)$
and, provided that the boundary conditions are chosen in the appropriate
trace space, such a regularity extends up to the boundary, namely
$u\in C^{\l,\alpha}\left(\overline{\Omega};\R\right)$. In the case
$\l=0$, the H\"older exponent $\alpha$ is for example the one given
by the De Giorgi--Nash--Moser theorem.

This paper focuses on how to enforce pointwise constraints on the
solutions of (\ref{eq:PDE}). Our motivation for studying such a question
comes from hybrid imaging. Hybrid, or multi-physics, imaging problems
are a type of parameter identification problems that in many cases
involve the reconstruction of the coefficients of a PDE from the knowledge
of some internal functional of its solutions \cite{AMMARI-2008,WIDLAK-SCHERZER-2012,BAL-2012,kuchment-2012,ammari-2016,alberti-capdeboscq-2018}.

Amongst all these constraints, the most ubiquitous one is the non-vanishing
Jacobian problem. It can be reworded as follows: given $L$ as in
(\ref{eq:PDE}) and a compact set $K\subseteq\Omega$, how can one
choose boundary conditions $g_{1},\ldots,g_{N}$ such that
\begin{equation}
\rank\left(\nabla u_{1},\ldots,\nabla u_{N}\right)=d\text{ everywhere in }K,\label{eq:JacCon}
\end{equation}
where 
\[
\left\{ \begin{array}{ll}
Lu_{i}=0 & \text{in }\ensuremath{\Omega},\\
u_{i}=g_{i} & \text{on }\ensuremath{\bo},
\end{array}\right.\qquad i=1,\ldots,N?
\]
The difficulty here is that, apart from the fact that $a$, $b$,
$c$ and $q$ are relatively smooth and coercive, nothing is known
about these coefficients, which are the unknowns of the inverse problem.

When $d=2$, $b=0,c=0$ and $q=0$, that is, $L$ is simply 
\[
Lu=-\div(a\nabla u),
\]
it turns out that the Rad\'o--Kneser--Choquet Theorem can be extended
to this setting (without regularity assumptions) \cite{alessandrini-magnanini-1994,ALESSANDRINI-NESI-01,MR3364666,ALESSANDRINI-NESI-2015}.
Only two boundary conditions, independent of the matrix valued function
$a$, are required for the constraint to be satisfied globally. This
result cannot be extended to higher dimensions \cite{WOOD-1991,LAUGESEN-96,CAPDEBOSCQ-15,alberti-bal-dicristo-2016,alberti-capdeboscq-2018},
even locally: it is not possible to find suitable boundary conditions
independently of the (unknown) coefficient. For more general models,
when $b,c$ or $q$ are not null, such as the Helmholtz equation,
no solace can be found in any dimension, since the Rad\'o--Kneser--Choquet
Theorem, whose proof uses the maximum principle, does not apply.

One is therefore drawn to ask whether using a large number of boundary
conditions would help. Again, some counter-examples can be derived
to the most optimistic claims \cite[Corollary~6.18]{alberti-capdeboscq-2018};
nevertheless, it is possible to construct open sets of boundary conditions
valid for open sets of parameters for the relevant elliptic operator
$L$. Two main strategies have been used to achieve this goal: complex
geometrical optics (CGO) solutions \cite{TRIKI-2010,BAL-UHLMANN-2010,BAL-REN-UHLMANN-ZHOU-2011,BAL-REN-2011,AMMARI-CAPDEBOSCQ-DEGOURNAY-ROZANOVA-TRIKI-2011,KOCYIGIT-2012,BAL-MONARD-2012,BAL-2012,2012-MONARD-BAL-2,BAL-BONNETIER-MONARD-TRIKI-2013,BAL-UHLMANN-2013,ammari-2013,BAL-SCHOTLAND-2014,bal-guo-monard-2014}
and the Runge approximation property \cite{bal-guo-monard-2014,MONARD-BAL-2013,BAL-UHLMANN-2013,BAL-SCHOTLAND-2014}.
Other approaches based on frequency variations \cite{alberti-2013,alberti-2015,alberti-2015b,alberti-genericity,capalb-analytic,alberti-ammari-ruan-2016,AMMARI-GIOVANGIGLI-NGUYEN-SEO}
or dynamical systems \cite{BAL20131357} were also developed: these
are not discussed here.

CGO solutions are only available for isotropic coefficients $a$,
that is, $a=\gamma I_{d}$ where $\gamma$ is a real-valued function.
The CGO solution method provides a non-vanishing Jacobian globally
inside the domain for a suitable choice of ($d$ complex-valued) boundary
conditions. This approach requires high regularity assumptions on
the coefficients. On the other hand, the Runge approximation property
holds provided that the unique continuation property holds \cite{LAX-1956},
such a property being enjoyed by a much larger class of problems \cite{alessandrini-rondi-rosset-vessella-2009}.
A drawback is that the argument is local, applied on a covering of
the domain by small balls, and so many boundary conditions are needed.
Further, while CGO solutions are constructed (depending on the coefficients),
the Runge approximation provides an existence result of suitable solutions,
but not a constructive method to derive them.

In this work, we combine a Whitney projection argument \cite{whitney-1937},
as described in \cite{Greene-Wu-1975a,Greene-Wu-1975b}, with the
Runge approximation. Not only do we provide an explicit bound on the
number of boundary conditions to be considered, but we also obtain
that these constitute an open and dense set. For instance, the set
of $2d$ boundary conditions such that (\ref{eq:JacCon}) is satisfied
is open and dense in $H^{1/2}(\bo;\R)^{2d}$. Our result applies to
more general constraints than (\ref{eq:JacCon}), so that it is in
particular applicable to a variety of imaging problems (see section~\ref{sec:MR}
for details). Our result confirms what has been observed in numerical
simulations in the setting of scalar (isotropic) diffusion coefficients,
where good reconstructions are obtained for a relatively small set
of boundary conditions \cite{BAL-MONARD-2012,MONARD-BAL-2013,BAL20131357,BAL-UHLMANN-2013}.
After the first version of this manuscript was published, we were
made aware of the recent preprint \cite{cekic2018calder}, where similar
techniques are used for the fractional Calder\'on problem.

This paper is structured as follows. In section~\ref{sec:MR}, we
state our main results and discuss some open problems. Section~\ref{sec:runge}
is devoted to the Runge approximation property. Finally, in section~\ref{sec:Proofs}
we provide the proof of the main result.

\section{\label{sec:MR}Main results}

Let $K\subseteq\overline{\Omega}$ be a smooth compact set and
\[
\zeta\colon C^{{\l,\alpha}}(K)\to C^{{0,\alpha}}(K)^{n}
\]
be a continuous linear map, with $n\ge1$. Let $\H(K)$ denote the
set of solutions to (\ref{eq:PDE}) that are smooth in $K$, namely
\begin{equation}
\H(K)=\left\{ u\in C^{{\l,\alpha}}(K)\cap H^{1}(\Omega):Lu=0\;\text{in \ensuremath{\Omega}}\right\} ,\label{eq:defHK}
\end{equation}
equipped with the norm $\|u\|_{\H(K)}=\|u\|_{C^{{\l,\alpha}}(K)}+\|u\|_{H^{1}(\Omega)}$.
We are interested in solutions $u_{i}\in\H(K)$ satisfying the constraint
\begin{equation}
\det\begin{bmatrix}\zeta\left(u_{1}\right)\\
\vdots\\
\zeta\left(u_{n}\right)
\end{bmatrix}(x)\neq0,\label{eq:constraint}
\end{equation}
in $K$, locally or globally.
\begin{example}
\label{exa:Constraints-of-the}Constraints of the form (\ref{eq:constraint})
appear in various problems.
\begin{itemize}
\item When $n=1$, $\l={0}$ and $\zeta\left(u\right)=u$, the constraint
corresponds to avoiding nodal points, namely $u_{1}(x)\neq0$. This
is useful whenever a division by $u_{1}$ is required.
\item When $n=d$, $\l={1}$, and $\zeta\left(u\right)=\nabla u$ (taken
as row vector), the constraint imposes a non-vanishing Jacobian. In
that case, $\left(u_{1},\ldots,u_{d}\right)$ defines a local $C^{2}$
diffeomorphism. This is the case discussed in the introduction.
\item When $n=d+1$, $\l=1$, and $\zeta\left(u\right)=\begin{bmatrix}u & \nabla u\end{bmatrix}$
(taken as a row vector) the constraint imposes a non-vanishing ``augmented''
Jacobian. The additional potential $u$ may represent a scaled time
derivative, in a time harmonic model. This constraint may also correspond
to the non-vanishing Jacobian for $v_{i}=\frac{u_{i+1}}{u_{1}}$,
$i=1,\dots,d$.
\end{itemize}
Such constraints appear in quantitative photoacoustic tomography \cite{bal-2012-arxiv,BAL-UHLMANN-2013,BU-IP-10,BAL-REN-2011},
in quantitative thermoacoustic tomography \cite{BAL-REN-UHLMANN-ZHOU-2011,ammari-2013,alberti-2015b},
in acousto-electric tomography (also known as electrical impedance
tomography by elastic deformation or ultrasound modulated electrical
impedance tomography) \cite{ABCTF-SIAP-08,CAPDEBOSCQ-FEHRENBACH-DEGOURNAY-KAVIAN-09,2012-MONARD-BAL-2,BAL-BONNETIER-MONARD-TRIKI-2013,KS-IP-12,MONARD-BAL-2013,kuchment-2015},
in microwave imaging by elastic deformation \cite{AMMARI-CAPDEBOSCQ-DEGOURNAY-ROZANOVA-TRIKI-2011,alberti-2013,alessandrini-2014,TRIKI-2010},
in current density imaging \cite{hasanov-2004,lee-2004,NTT-Rev-11,bal-guo-monard-2014},
in dynamic elastography \cite{bal-2012-arxiv,bal-uhlmann-reconstructions-2012,BAL-UHLMANN-2013}
and in other hybrid imaging modalities. We refer to \cite{alberti-capdeboscq-2018}
for additional methods and further explanations on some of the models
we have mentioned.
\end{example}

We introduce the following notation.
\begin{defn}
The \emph{candidate set} $\mathcal{C}(K)$ is the set of all $x\in K$
for which there exist $\ensuremath{u_{1},\dots,u_{n}\in\H(K)}$ so
that 
\[
\det\begin{bmatrix}\zeta\left(u_{1}\right)\\
\vdots\\
\zeta\left(u_{n}\right)
\end{bmatrix}(x)\neq0.
\]
The \emph{admissible set} $\mathcal{E}(K)$ is the set of all $u=(u_{1},\dots,u_{\left[\frac{d+n}{\alpha}\right]})\in\H(K)^{\left[\frac{d+n}{\alpha}\right]}$
such that
\[
\sum_{i_{1,}\dots,i_{n}=1}^{\left[\frac{d+n}{\alpha}\right]}|\det\begin{bmatrix}\zeta\left(u_{i_{1}}\right)\\
\vdots\\
\zeta\left(u_{i_{n}}\right)
\end{bmatrix}(x)|>0,\text{ for all }x\in K.
\]
Here, $\left[\frac{d+n}{\alpha}\right]=\max\left\{ N\in\N:N\le\frac{d+n}{\alpha}\right\} $
denotes the integer part of $\frac{d+n}{\alpha}$.
\end{defn}

\begin{rem}
In other words, $u\in\H(K)^{\left[\frac{d+n}{\alpha}\right]}$ belongs
to $\mathcal{E}(K)$ if and only if for every $x\in K$ there exist
$i_{1},\dots,i_{n}\in\{1,\dots,\left[\frac{d+n}{\alpha}\right]\}$
such that
\[
\det\begin{bmatrix}\zeta\left(u_{i_{1}}\right)\\
\vdots\\
\zeta\left(u_{i_{n}}\right)
\end{bmatrix}(x)\neq0,
\]
namely, if and only if the desired constraint is satisfied everywhere
in $K$ for a suitable subset of the solutions $u_{1},\dots,u_{\left[\frac{d+n}{\alpha}\right]}$.
The candidate set is the subset of $K$ where satisfying the constraint
pointwise is possible at all. If $\mathcal{C}(K)\neq K$, then $\mathcal{E}\left(K\right)$
is empty.
\end{rem}

\begin{rem}
If the coefficients of the PDE are smooth enough, so that $\alpha>\frac{d+n}{d+n+1}$,
then the number of solutions is $d+n=\left[\frac{d+n}{\alpha}\right]$.
\end{rem}

In this general setting, we have the following result.
\begin{thm}
\label{thm:general}Take a smooth compact set $K\subseteq\overline{\Omega}$.
Assume that the candidate set satisfies 
\begin{equation}
\mathcal{C}\left(K\right)=K.\label{eq:assumption-runge}
\end{equation}
Then the admissible set $\mathcal{E}\left(K\right)$ is open and dense
in $\H\left(K\right){}^{\left[\frac{d+n}{\alpha}\right]}$.
\end{thm}

\begin{rem}
Note that, since a finite intersection of open and dense sets is open
and dense, the result immediately extends to the case when finitely
many constraints are imposed simultaneously.
\end{rem}

In section~\ref{sec:runge} we observe, using the Runge Approximation
Property, that assumption (\ref{eq:assumption-runge}) is satisfied
for a large class of examples, since $\mathcal{C}(\overline{\Omega})=\overline{\Omega}$.

Our initial focus was on boundary value problems, since such problems
are relevant for non-invasive imaging methods, as explained in the
introduction. The following corollary is a rewording of our result
for boundary value problems.
\begin{cor}
\label{cor:general}Take a compact set $K\subseteq\Omega$. Suppose
that for every $g\in H^{1/2}(\bo)$ the problem
\begin{equation}
\left\{ \begin{array}{ll}
Lu=0 & \text{in }\ensuremath{\Omega},\\
u=g & \text{on }\ensuremath{\bo},
\end{array}\right.\label{eq:PDE with boundary}
\end{equation}
admits a unique solution $u^{g}\in H^{1}(\Omega).$ If (\ref{eq:assumption-runge})
holds true, then the set
\[
\left\{ \left(g_{1},\dots,g_{\left[\frac{d+n}{\alpha}\right]}\right)\in H^{1/2}(\bo)^{\left[\frac{d+n}{\alpha}\right]}:\left(u^{g_{1}},\dots,u^{g_{\left[\frac{d+n}{\alpha}\right]}}\right)\in\E(K)\right\} 
\]
is open and dense in $H^{1/2}(\bo)^{\left[\frac{d+n}{\alpha}\right]}$.
\end{cor}

\begin{proof}
Consider the map 
\[
\psi\colon H^{1/2}(\bo)^{\left[\frac{d+n}{\alpha}\right]}\to H^{1}(\Omega)^{\left[\frac{d+n}{\alpha}\right]},\qquad\left(g_{1},\dots,g_{\left[\frac{d+n}{\alpha}\right]}\right)\mapsto\left(u^{g_{1}},\dots,u^{g_{\left[\frac{d+n}{\alpha}\right]}}\right),
\]
where $u^{g_{i}}$ is defined by (\ref{eq:PDE with boundary}). Because
problem (\ref{eq:PDE with boundary}) is well-posed, we have $\|u^{g}\|_{H^{1}(\Omega)}\le C(L)\|g\|_{H^{1/2}(\bo)}$
for every $g\in H^{1/2}(\bo)$. Further, because of our outstanding
regularity assumptions on the coefficients (\ref{eq:reg}) we also
have $\|u^{g}\|_{C^{{\l,\alpha}}(K)}\le C(L)\|g\|_{H^{1/2}(\bo)}$.
This shows that the map $\psi\colon H^{1/2}(\bo)^{\left[\frac{d+n}{\alpha}\right]}\to\mathcal{H}\left(K\right)^{\left[\frac{d+n}{\alpha}\right]}$
is continuous. Its inverse is given by the trace operator acting component-wise,
and is also continuous. In other words, $\psi$ is an isomorphism,
and the result immediately follows from Theorem~\ref{thm:general},
since the set under consideration is $\psi^{-1}(\E(K))$.
\end{proof}
\begin{rem}
Corollary~\ref{cor:general} was stated for simplicity only if $K$
is a proper subset of $\Omega$. When $K$ touches $\bo$, for instance
if $K=\overline{\Omega}$, the same result holds, provided that $H^{1/2}(\bo)$
is replaced with a suitable trace space consisting of smoother functions.
\end{rem}

\subsection*{Future perspectives}

The regularity assumptions we made are important in all generality,
because we use a Unique Continuation Principle argument. For a specific
problem, with a given geometry and/or coefficient structure, appropriate
extension can often be envisioned (see e.g.\  \cite{BALL-CAPDEBOSCQ-TSERING-2012}
for a strategy on how to handle a large class of piecewise regular
coefficients). We have limited ourselves to elliptic PDE with real
coefficients. Considering the case of complex valued coefficients
(which appear in thermo-acoustic tomography \cite{BAL-REN-UHLMANN-ZHOU-2011,ammari-2013})
is a natural extension of this work. Maxwell's equation \cite{seo-kim-etal-2012,BAL-GUO-2013,bal-zhou-2014,alberti-2015,alberti-2015b}
and linear elasticity \cite{barbone-etal-2010,mclaughlin-zhang-manduca-2010,lai-2014,bal-bellis-imperiale-monard-2014,bal-monard-uhlmann-2015}
are not considered here and are natural frameworks where this method
could be applied. Finally, note that while a rough description of
our result could be that a ``random'' choice of $\left[\frac{d+n}{\alpha}\right]$
boundary condition suffices, we have not established such a claim.
It would be interesting to move from an open and dense set of admissible
boundary conditions to a random choice of boundary conditions with
high probability (or indeed probability $1$).

\section{\label{sec:runge}The Runge approximation property and assumption
(\ref{eq:assumption-runge})\label{sec:The-Runge-approximation}}

For simplicity of exposition, in this section we restrict ourselves
to considering only the constraints associated to the maps $\zeta$
given in Example~\ref{exa:Constraints-of-the}, namely:
\begin{itemize}
\item $n=1$, $\l={0}$, $\zeta\left(u\right)=u$;
\item $n=d$, $\l={1}$, $\zeta\left(u\right)=\nabla u$;
\item or $n=d+1$, $\l={1}$, $\zeta\left(u\right)=\begin{bmatrix}u & \nabla u\end{bmatrix}$.
\end{itemize}
However, with minor modifications to the argument, many other constraints
can be considered, since this approach is very general. The main tool
to satisfy (\ref{eq:assumption-runge}), namely to show that there
always exist global solutions satisfying the desired constraints locally,
is the following result: it is sufficient to build suitable solutions
of the PDE with constant coefficients, and without lower order terms.
\begin{prop}
\label{prop:runge}Let $L$ be the elliptic operator defined in (\ref{eq:PDE}).
In addition to (\ref{eq:coercive}), (\ref{eq:reg}) if $\l=1$ and
(\ref{eq:reg0}) if $\l=0$, assume that $a(x)$ is a symmetric matrix
for every $x\in\overline{\Omega}$ and, if $d\ge3$, $a\in C^{0,1}\left(\overline{\Omega};\R^{d\times d}\right)$.
Take $x_{0}\in\overline{\Omega}$. Let $r>0$ and $u_{1},\dots,u_{n}\in C^{\l,\alpha}(\overline{\Omega};\R)$
be solutions to the constant coefficient problem
\[
-\div\left(a\left(x_{0}\right)\nabla u_{i}\right)=0\quad\text{in }B(x_{0},r),\qquad i=1,\dots,n.
\]
If

\begin{equation}
\det\begin{bmatrix}\zeta\left(u_{1}\right)\\
\vdots\\
\zeta\left(u_{n}\right)
\end{bmatrix}(x_{0})\neq0,\label{eq:x_0}
\end{equation}
then $x_{0}\in\mathcal{C}(\overline{\Omega})$.
\end{prop}

\begin{rem}
This result can in some cases be extended to operators $L$ with piecewise
Lipschitz coefficients with possibly countably many pieces, following
the strategy given in \cite{BALL-CAPDEBOSCQ-TSERING-2012}.
\end{rem}

\begin{proof}
This result, even though not in this exact form, was first derived
in \cite{BAL-UHLMANN-2013}, and later discussed in \cite[Section~7.3]{alberti-capdeboscq-2018}
(only in the case $x_{0}\in\Omega$). Here we provide only a sketch
of the proof in order to highlight the main features; the reader is
referred to the references mentioned for the details of the argument.

The proof is split into three steps.

\emph{Step 1: approximation of $u_{i}$ with local solutions $v_{i}$
to $Lv_{i}=0$. }Using standard elliptic regularity estimates, it
is possible to find $\tilde{r}\in(0,r]$ and $v_{i}\in H^{1}(B(x_{0},\tilde{r})\cap\Omega)$
such that $Lv_{i}=0$ in $B(x_{0},\tilde{r})\cap\Omega$ and $\|u_{i}-v_{i}\|_{C^{1}\left(\overline{B(x_{0},\tilde{r})}\cap\overline{\Omega}\right)}$
is arbitrarily small (provided that $\tilde{r}$ is chosen small enough).
It is worth observing that, even if in \cite{BAL-UHLMANN-2013} the
lower order terms are kept in the PDE with constant coefficients,
that is not needed \cite[Proposition~7.10]{alberti-capdeboscq-2018}.

\emph{Step 2: approximation of $v_{i}$ with global solutions $w_{i}$
to $Lw_{i}=0$. }Thanks to the regularity assumptions on the coefficients,
the elliptic operator $L$ enjoys the unique continuation property
\cite{alessandrini-rondi-rosset-vessella-2009}. This is equivalent
to the Runge approximation property \cite{LAX-1956}, by which it
is possible to approximate local solutions with global solutions.
Thus, in our setting, there exist $w_{i}\in H^{1}(\Omega)$ solutions
to $Lw_{i}=0$ in $\Omega$ such that $\|w_{i}-v_{i}\|_{H^{1}(B(x_{0},\tilde{r})\cap\Omega)}$
is arbitrarily small. By elliptic regularity, we can ensure that $\|w_{i}-v_{i}\|_{C^{1}\left(\overline{B(x_{0},\tilde{r}/2)}\cap\overline{\Omega}\right)}$
is arbitrarily small too.

\emph{Step 3: $(w_{1},\dots,w_{n})$ satisfy the constraint in $x_{0}$.
}Combining the previous steps, we have that $\|u_{i}-w_{i}\|_{C^{1}\left(\overline{B(x_{0},\tilde{r}/2)}\cap\overline{\Omega}\right)}$
is arbitrarily small. For the maps $\zeta$ considered above, this
immediately implies that $|\zeta(u_{i})-\zeta(w_{i})|(x_{0})$ is
arbitrarily small. Thus, by (\ref{eq:x_0}), $\tilde{r}$ and $w_{i}$
may be chosen in such a way that
\[
\det\begin{bmatrix}\zeta\left(w_{1}\right)\\
\vdots\\
\zeta\left(w_{n}\right)
\end{bmatrix}(x_{0})\neq0,
\]
which shows that $x_{0}\in\mathcal{C}(\overline{\Omega})$.
\end{proof}
Let us now verify that for the maps $\zeta$ mentioned above, we always
have $\mathcal{C}(\overline{\Omega})=\overline{\Omega}$; in other
words, the assumptions of Theorem~\ref{thm:general} are satisfied
with $K=\overline{\Omega}$.
\begin{cor}
\label{cor:runge}Let $L$ be the elliptic operator defined in (\ref{eq:PDE}).
In addition to (\ref{eq:coercive}), (\ref{eq:reg}) if $\l=1$ and
(\ref{eq:reg0}) if $\l=0$, assume that $a(x)$ is a symmetric matrix
for every $x\in\overline{\Omega}$ and, if $d\ge3$, $a\in C^{0,1}\left(\overline{\Omega};\R^{d\times d}\right)$.
If $\zeta$ is one of the maps considered in Example~\ref{exa:Constraints-of-the},
then $\mathcal{C}(K)=K$ for any $K\subseteq\overline{\Omega}$.
\end{cor}

\begin{proof}
We consider the three constraints separately:
\begin{itemize}
\item $n=1$, $\l={0}$, $\zeta\left(u\right)=u$: set $u_{1}=1$.
\item $n=d$, $\l={1}$, $\zeta\left(u\right)=\nabla u$: set $u_{1}=x_{1},\dots,u_{n}=x_{n}$.
\item $n=d+1$, $\l={1}$, $\zeta\left(u\right)=\begin{bmatrix}u & \nabla u\end{bmatrix}$:
set $u_{1}=1,u_{2}=x_{1},\dots,u_{n+1}=x_{n}$.
\end{itemize}
Given $x_{0}\in\overline{\Omega}$, for any $a(x_{0})$, there holds
\[
-\div(a(x_{0})\nabla u_{i})=0\quad\text{in }\R^{d},\qquad i=1,\dots,n,
\]
and 
\[
\det\begin{bmatrix}\zeta\left(u_{1}\right)\\
\vdots\\
\zeta\left(u_{n}\right)
\end{bmatrix}(x_{0})\neq0.
\]
The conclusion follows from Proposition~\ref{prop:runge}.
\end{proof}

\section{\label{sec:Proofs}Proof of Theorem~\ref{thm:general}}

We need two lemmata. For $k\ge2$ and $a\in\R^{k-1}$ let $P_{a}\colon\R^{k}\to\R^{k-1}$
denote the linear map given by
\[
P_{a}(y)=\left(y_{1}-a_{1}y_{k},\dots,y_{k-1}-a_{k-1}y_{k}\right).
\]
In the following, we shall identify the matrices in $\R^{k\times n}$
with $k$ rows and $n$ columns with the linear maps from $\R^{n}$
into $\R^{k}$. We shall denote the Lebesgue measure in $\R^{m}$
by $|\cdot|_{m}$.
\begin{lem}
\label{lem:whitney}Take a smooth and compact set $K\subseteq\R^{d}$
and a positive integer $k>{\frac{d+n}{\alpha}}$. Let $F\colon K\to\R^{k\times n}$
be of class $C^{{0,\alpha}}$ and such that $F_{x}\colon\R^{n}\to\R^{k}$
is injective for all $x\in K$. Let $G\subseteq\R^{k-1}$ be the set
of those $a\in\R^{k-1}$ for which $P_{a}\circ F_{x}\colon\R^{n}\to\R^{k-1}$
is injective for all $x\in K$, namely
\[
G=\bigcap_{x\in K}\left\{ a\in\R^{k-1}:P_{a}\circ F_{x}\colon\R^{n}\to\R^{k-1}\;\text{is injective}\right\} .
\]
Then $|\R^{k-1}\setminus G|_{k-1}=0$.
\end{lem}

\begin{proof}
Note that $\ker P_{a}=\span\{(a_{1},\dots,a_{k-1},1)\}$. Thus, since
$F_{x}$ is injective, we have for $x\in K$
\[
\begin{split}P_{a}\circ F_{x}\;\text{is injective} & \iff\ran F_{x}\cap\ker P_{a}=\{0\}\\
 & \iff(a_{1},\dots,a_{k-1},1)\notin\ran F_{x}.
\end{split}
\]
Then, $a\in\R^{k-1}\setminus G$ if and only if there exists $x\in K$
such that $(a_{1},\dots,a_{k-1},1)\in\ran F_{x}$, namely $(a_{1},\dots,a_{k-1},1)\in\cup_{x\in K}\ran F_{x}$.
Therefore, using the projection $\pi\colon\R^{k}\to\R^{k-1}$, $(b_{1},\dots,b_{k})\mapsto(b_{1},\dots,b_{k-1})$,
we can express $\R^{k-1}\setminus G$ as
\[
\R^{k-1}\setminus G=\pi\left(B\right),\qquad B=\biggl(\bigcup_{x\in K}\ran F_{x}\biggr)\cap\left\{ b\in\R^{k}:b_{k}=1\right\} .
\]
Hence, it remains to prove that
\[
\mathcal{H}^{k-1}\left(B\right)=0,
\]
where $\mathcal{H}^{k-1}$ denotes the Hausdorff measure of dimension
$k-1$.

Consider the map
\[
f\colon K\times\R^{n}\to\R^{k},\qquad\left(x,v\right)\mapsto F_{x}v.
\]
By construction, $\ran f=\bigcup_{x\in K}\ran F_{x}$. Note that $f$
is of class $C^{0,\alpha}$ and that $\dim(K\times\R^{n})=d+n<\alpha k$,
so that $\H^{\alpha k}(K\times\R^{n})=0$. By elementary properties
of Hausdorff measures \cite[section~4.1]{alberti-2012-sard}, we have
$\H^{k}(\bigcup_{x\in K}\ran F_{x})=\H^{k}(\ran f)=0$. By linearity
of $v\mapsto F_{x}v$, the set $\bigcup_{x\in K}\ran F_{x}$ is closed
under scalar multiplication, and so $\bigcup_{x\in K}\ran F_{x}\supseteq\R_{+}\cdot B$,
which implies
\[
\H^{k}(\R_{+}\cdot B)=0.
\]
Using the change of variables formula and Tonelli theorem, we deduce
$\H^{k-1}(B)=0$, as desired.
\end{proof}
\begin{lem}
\label{lem:2}Take a smooth and compact set $K\subseteq\overline{\Omega}$
and a positive integer $k>{\frac{d+n}{\alpha}}$. Let $u_{1},\dots,u_{k}\in\H(K)$
be such that
\[
\rank\begin{bmatrix}\zeta\left(u_{1}\right)\\
\vdots\\
\zeta\left(u_{k}\right)
\end{bmatrix}(x)=n,\qquad x\in K.
\]
Let $G\subseteq\R^{k-1}$ be the set of those $a\in\R^{k-1}$ such
that

\[
\rank\begin{bmatrix}\zeta\left(u_{1}-a_{1}u_{k}\right)\\
\zeta\left(u_{2}-a_{2}u_{k}\right)\\
\vdots\\
\zeta\left(u_{k-1}-a_{k-1}u_{k}\right)
\end{bmatrix}(x)=n,\qquad x\in K.
\]
Then $|\R^{k-1}\setminus G|_{k-1}=0$.
\end{lem}

\begin{proof}
Let $F\colon K\to\R^{k\times n}$ be the map of class $C^{{0,\alpha}}$
defined by
\[
F_{x}=\begin{bmatrix}\zeta\left(u_{1}\right)\\
\vdots\\
\zeta\left(u_{k}\right)
\end{bmatrix}(x)\colon\R^{n}\to\R^{k}.
\]
By assumption, we have that $F_{x}$ is injective for all $x\in K$.
Observe that, by linearity of $\zeta,$ we have
\[
P_{a}\circ F_{x}=\begin{bmatrix}\zeta\left(u_{1}\right)(x)-a_{1}\zeta\left(u_{k}\right)(x)\\
\vdots\\
\zeta\left(u_{k-1}\right)(x)-a_{k-1}\zeta\left(u_{k}\right)(x)
\end{bmatrix}=\begin{bmatrix}\zeta\left(u_{1}-a_{1}u_{k}\right)\\
\vdots\\
\zeta\left(u_{k-1}-a_{k-1}u_{k}\right)
\end{bmatrix}(x)
\]
and so the conclusion immediately follows by Lemma~\ref{lem:whitney}.
\end{proof}
We are now ready to prove Theorem~\ref{thm:general}.
\begin{proof}[\emph{Proof of }Theorem~\ref{thm:general}]
\emph{}

\emph{Step 1: $\mathcal{E}(K)$ is open in $\H(K)^{\left[\frac{d+n}{\alpha}\right]}$.}

Take $u\in\E(K)$. By definition, we have
\[
\sum_{i_{1,}\dots,i_{n}=1}^{\left[\frac{d+n}{\alpha}\right]}|\det\begin{bmatrix}\zeta\left(u_{i_{1}}\right)\\
\vdots\\
\zeta\left(u_{i_{n}}\right)
\end{bmatrix}(x)|>0,\qquad x\in K.
\]
Since $\zeta(u_{i_{j}})$ are $C^{{0,\alpha}}$ maps, and in particular
continuous, and $K$ is compact, we have that
\[
\sum_{i_{1,}\dots,i_{n}=1}^{\left[\frac{d+n}{\alpha}\right]}|\det\begin{bmatrix}\zeta\left(u_{i_{1}}\right)\\
\vdots\\
\zeta\left(u_{i_{n}}\right)
\end{bmatrix}(x)|\ge C,\qquad x\in K,
\]
for some constant $C>0$. Finally, since the map $\zeta$ is continuous
itself, if $v\in\H(K)^{\left[\frac{d+n}{\alpha}\right]}$ is chosen
close enough to $u$ we have
\[
\sum_{i_{1,}\dots,i_{n}=1}^{\left[\frac{d+n}{\alpha}\right]}|\det\begin{bmatrix}\zeta\left(v_{i_{1}}\right)\\
\vdots\\
\zeta\left(v_{i_{n}}\right)
\end{bmatrix}(x)|\ge\frac{C}{2},\qquad x\in K,
\]
which implies $v\in\E(K)$. This concludes the first step.

\emph{Step 2: $\mathcal{E}(K)$ is dense in $\H(K)^{\left[\frac{d+n}{\alpha}\right]}$.}

Take $H=(h_{1},\dots,h_{\left[\frac{d+n}{\alpha}\right]})\in\H(K)^{\left[\frac{d+n}{\alpha}\right]}$.
By assumption, for all $x\in K$ there exist $u_{1,x},\dots,u_{n,x}\in\H(K)$
such that
\[
|\det\begin{bmatrix}\zeta\left(u_{1,x}\right)\\
\vdots\\
\zeta\left(u_{n,x}\right)
\end{bmatrix}(x)|>0.
\]
By continuity of $\zeta\left(u_{i,x}\right)$, there exist neighbourhoods
$U^{x}\ni x$ such that
\[
|\det\begin{bmatrix}\zeta\left(u_{1,x}\right)\\
\vdots\\
\zeta\left(u_{n,x}\right)
\end{bmatrix}(y)|>0,\qquad y\in U^{x}\cap K.
\]
Since $K\subseteq\cup_{x\in K}U^{x}$, by compactness there exist
$x_{1},\dots,x_{N}\in K$ such that $K\subseteq\cup_{j=1}^{N}U^{x_{j}}$.
Thus, for all $x\in K$ there exists $j=1,\dots,N$ such that
\[
|\det\begin{bmatrix}\zeta\left(u_{1,x_{j}}\right)\\
\vdots\\
\zeta\left(u_{n,x_{j}}\right)
\end{bmatrix}(x)|>0.
\]
Consider all the $M=nN$ corresponding solutions
\[
(u_{1},\dots,u_{M})=\left(u_{1,x_{1}},\dots,u_{n,x_{1}},u_{1,x_{2}},\dots,u_{n,x_{2}},\dots,u_{1,x_{N}},\dots,u_{n,x_{N}}\right),
\]
 so that
\[
\rank\begin{bmatrix}\zeta(u_{1})\\
\vdots\\
\zeta(u_{M})
\end{bmatrix}(x)=n,\qquad x\in K.
\]
In particular, we have
\[
\rank\begin{bmatrix}\zeta(h_{1})\\
\vdots\\
\zeta(h_{\left[\frac{d+n}{\alpha}\right]})\\
\zeta(u_{1})\\
\vdots\\
\zeta(u_{M})
\end{bmatrix}(x)=n,\qquad x\in K.
\]

By Lemma~\ref{lem:2}, since $k=\left[\frac{d+n}{\alpha}\right]+M>\left[\frac{d+n}{\alpha}\right]$,
we have $k>\frac{d+n}{\alpha}$, and so for almost every $a\in\R^{k-1}$
we have
\[
\rank\begin{bmatrix}\zeta(h_{1}-a_{1}u_{M})\\
\vdots\\
\zeta(h_{\left[\frac{d+n}{\alpha}\right]}-a_{\left[\frac{d+n}{\alpha}\right]}u_{M})\\
\zeta(u_{1}-a_{\left[\frac{d+n}{\alpha}\right]+1}u_{M})\\
\vdots\\
\zeta(u_{M-1}-a_{k-1}u_{M})
\end{bmatrix}(x)=n,\qquad x\in K.
\]
Repeating this argument $M$ times (as long as $k>\left[\frac{d+n}{\alpha}\right]$)
with very small weights $a$, we obtain that there exist $\xi_{i,j}\in\R$
($i=1,\dots,\left[\frac{d+n}{\alpha}\right]$, $j=1,\dots,M$) which
can be chosen arbitrarily small such that
\[
\rank\begin{bmatrix}\zeta(h_{1}-\xi_{1,j}u_{j})\\
\vdots\\
\zeta(h_{\left[\frac{d+n}{\alpha}\right]}-\xi_{\left[\frac{d+n}{\alpha}\right],j}u_{j})
\end{bmatrix}(x)=n,\qquad x\in K,
\]
where we used Einstein summation convention of repeated indices. This
implies that
\[
(h_{1}-\xi_{1,j}u_{j},\dots,h_{\left[\frac{d+n}{\alpha}\right]}-\xi_{\left[\frac{d+n}{\alpha}\right],j}u_{j})\in\E(K),
\]
which, since the weights $\xi_{i,j}$ are chosen arbitrarily small,
concludes the proof.
\end{proof}
\bibliographystyle{abbrv}
\bibliography{rungebiblio}

\end{document}